\documentclass[12pt]{amsart}
\usepackage{amssymb}
\usepackage{verbatim}
\usepackage{hyperref}
\usepackage{color}

\begin{document}

\newtheorem{theorem}{Theorem}    
\newtheorem{proposition}[theorem]{Proposition}
\newtheorem{conjecture}[theorem]{Conjecture}
\def\theconjecture{\unskip}
\newtheorem{corollary}[theorem]{Corollary}
\newtheorem{lemma}[theorem]{Lemma}
\newtheorem{sublemma}[theorem]{Sublemma}
\newtheorem{observation}[theorem]{Observation}
\theoremstyle{definition}
\newtheorem{definition}{Definition}
\newtheorem{notation}[definition]{Notation}
\newtheorem{remark}[definition]{Remark}
\newtheorem{question}[definition]{Question}
\newtheorem{questions}[definition]{Questions}
\newtheorem{example}[definition]{Example}
\newtheorem{problem}[definition]{Problem}
\newtheorem{exercise}[definition]{Exercise}

\numberwithin{theorem}{section} \numberwithin{definition}{section}
\numberwithin{equation}{section}

\def\earrow{{\mathbf e}}
\def\rarrow{{\mathbf r}}
\def\uarrow{{\mathbf u}}
\def\varrow{{\mathbf V}}
\def\tpar{T_{\rm par}}
\def\apar{A_{\rm par}}

\def\reals{{\mathbb R}}
\def\torus{{\mathbb T}}
\def\heis{{\mathbb H}}
\def\integers{{\mathbb Z}}
\def\naturals{{\mathbb N}}
\def\complex{{\mathbb C}\/}
\def\distance{\operatorname{distance}\,}
\def\support{\operatorname{support}\,}
\def\dist{\operatorname{dist}\,}
\def\Span{\operatorname{span}\,}
\def\degree{\operatorname{degree}\,}
\def\kernel{\operatorname{kernel}\,}
\def\dim{\operatorname{dim}\,}
\def\codim{\operatorname{codim}}
\def\trace{\operatorname{trace\,}}
\def\Span{\operatorname{span}\,}
\def\dimension{\operatorname{dimension}\,}
\def\codimension{\operatorname{codimension}\,}
\def\nullspace{\scriptk}
\def\kernel{\operatorname{Ker}}
\def\ZZ{ {\mathbb Z} }
\def\p{\partial}
\def\rp{{ ^{-1} }}
\def\Re{\operatorname{Re\,} }
\def\Im{\operatorname{Im\,} }
\def\ov{\overline}
\def\eps{\varepsilon}
\def\lt{L^2}
\def\diver{\operatorname{div}}
\def\curl{\operatorname{curl}}
\def\etta{\eta}
\newcommand{\norm}[1]{ \|  #1 \|}
\def\expect{\mathbb E}
\def\bull{$\bullet$\ }
\def\C{\mathbb{C}}
\def\R{\mathbb{R}}
\def\Rn{{\mathbb{R}^n}}
\def\Sn{{{S}^{n-1}}}
\def\M{\mathbb{M}}
\def\N{\mathbb{N}}
\def\Q{{\mathbb{Q}}}
\def\Z{\mathbb{Z}}
\def\F{\mathcal{F}}
\def\L{\mathcal{L}}
\def\S{\mathcal{S}}
\def\supp{\operatorname{supp}}
\def\dist{\operatorname{dist}}
\def\essi{\operatornamewithlimits{ess\,inf}}
\def\esss{\operatornamewithlimits{ess\,sup}}
\def\xone{x_1}
\def\xtwo{x_2}
\def\xq{x_2+x_1^2}
\newcommand{\abr}[1]{ \langle  #1 \rangle}

\newcommand{\Norm}[1]{ \left\|  #1 \right\| }
\newcommand{\set}[1]{ \left\{ #1 \right\} }
\def\one{\mathbf 1}
\def\whole{\mathbf V}
\newcommand{\modulo}[2]{[#1]_{#2}}

\def\scriptf{{\mathcal F}}
\def\scriptg{{\mathcal G}}
\def\scriptm{{\mathcal M}}
\def\scriptb{{\mathcal B}}
\def\scriptc{{\mathcal C}}
\def\scriptt{{\mathcal T}}
\def\scripti{{\mathcal I}}
\def\scripte{{\mathcal E}}
\def\scriptv{{\mathcal V}}
\def\scriptw{{\mathcal W}}
\def\scriptu{{\mathcal U}}
\def\scriptS{{\mathcal S}}
\def\scripta{{\mathcal A}}
\def\scriptr{{\mathcal R}}
\def\scripto{{\mathcal O}}
\def\scripth{{\mathcal H}}
\def\scriptd{{\mathcal D}}
\def\scriptl{{\mathcal L}}
\def\scriptn{{\mathcal N}}
\def\scriptp{{\mathcal P}}
\def\scriptk{{\mathcal K}}
\def\frakv{{\mathfrak V}}

\title[Boundedness of sublinear operators]
{Boundedness of sublinear operators on weighted Morrey spaces
and applications}

\author[Zunwei Fu]{Zunwei Fu} 

\author[Shanzhen Lu]{Shanzhen Lu}

\author[Shaoguang Shi]{Shaoguang Shi$^{*}$}
%
\thanks{This work was partially supported by
 NSF of China (Grant Nos. 10901076, 10931001 and 11171345), NSF of Shandong Province (Grant No. ZR2012AQ026), Beijing Natural Science Foundation
(Grant No. 1102023), Program for Changjiang Scholars and Innovative
Research Team in University. This work was also supported by
the Key Laboratory of Mathematics and Complex System(Beijing Normal
University), Ministry of Education, China. \\ \indent $^{*}$Corresponding author.}

\address{Zunwei Fu\\Department of Mathematics\\Linyi
University \\
Linyi 276005\\
P. R. China } \email{zwfu@mail.bnu.edu.cn}

\address{Shanzhen Lu\\School of Mathematical
Sciences\\Beijing Normal University\\Beijing, 100875\\
P. R. China} \email{lusz@bnu.edu.cn}

\address{Shaoguang Shi\\Department of Mathematics\\
Linyi
University \\
Linyi 276005\\
P. R. China}
\email{shishaoguang@yahoo.com.cn}

\maketitle

\begin{abstract}
We study the boundedness of some sublinear
operators on weighted Morrey spaces under certain size conditions. These conditions are satisfied by most of the
operators in harmonic analysis, such as the Hardy-Littlewood maximal operator, Calder\'{o}n-Zygmund singular integral operator, Bochner-Riesz means at the critical index, oscillatory singular operators, singular integral
operators with oscillating kernels and so on. As applications, the regularity in weighted Morrey spaces of strong
solutions to nondivergence elliptic equations with VMO coefficients are established.
\end{abstract}

\section{Introduction and main results} 
As is well known that Morrey \cite{Mo} introduced the classical
Morrey spaces to investigate the local behavior of solutions to
second order elliptic partial differential equations(PDE). We
recall its definition as
$$
M_{p,q}(\mathbb{R}^{n})=\left\{f:\|f\|_{M_{p,q}(\mathbb{R}^{n})}=\sup_{B\subset
\mathbb{R}^{n}}\left(\frac{1}{|B|^{1-\frac{p}{q}}}\int_{B}|f(x)|^{p}dx\right)^{\frac{1}{p}}<\infty\right\},
$$
where $f\in L_{loc}^{p}(\mathbb{R}^{n})$ and $1\leq p\leq
q<\infty.$ Here and after, $B$ denotes any balls in
$\mathbb{R}^{n}$. $M_{p,q}(\mathbb{R}^{n})$ was an expansion of
$L^{p}(\mathbb{R}^{n})$ in the sense that
$M_{p,p}(\mathbb{R}^{n})=L^{p}(\mathbb{R}^{n})$. Morrey found that
many properties of solutions to PDE can be attributed to the
boundedness of some operators on Morrey spaces.

Maximal functions and singular integrals play a key role
in harmonic analysis since maximal functions could control crucial
quantitative information concerning the given functions, despite
their larger size, while singular integrals, Hilbert transform as
it's prototype, nowadays intimately connected with PDE, operator
theory and other fields.

Let $f\in L_{loc}(\mathbb{R}^{n})$. The Hardy-Littlewood(H-L)
maximal function of $f$ is defined by
$$Mf(x)=\sup_{B\ni x}\frac{1}{|B|}\int_{B}|f(y)|dy.$$

The Calder\'{o}n-Zygmund(C-Z) singular integral operator is
defined by
$$
Tf(x)=p.v.\int_{\mathbb{R}^{n}}K(x-y)f(y)\,dy,
$$
where $K$ is a C-Z kernel \cite{G}. Chiarenza and Frasca \cite{CF}
obtained the boundedness of H-L maximal function $Mf(x)$ and C-Z
singular integral operator $T$ on $M_{p,q}(\mathbb{R}^{n})$. For
some works on the boundedness for the multilinear C-Z singular
integral operators on Morrey type spaces, see e.g \cite{GT}.

Let $0<\alpha<n$. The fractional integral is defined by
$$I_{\alpha}f(x)=\int_{\mathbb{R}^{n}}\frac{f(y)}{|y-x|^{n-\alpha}}dy.$$
An early impetus to the study of fractional integrals originated
from the problem of fractional derivation, see e.g. \cite{AS} and
\cite{OS}. Besides it's contributions to harmonic analysis,
fractional integrals also play an essential role in many other
fields. The Hardy-Littlewood-Soblev inequality about fractional
integral is still an indispensable tool to establish time-space
estimates for the heat semigroup of nonlinear evolution equations,
for some of this work, see e.g. \cite{K}. In recent times, the
applications to Chaos and Fractal have become another motivation
to study fractional integrals, see e.g. \cite{KG} and \cite{LS}.
The boundedness of $I_{\alpha}$ on $M_{p,q}(\mathbb{R}^{n})$ was
first established by Adams in \cite{A}.

On the other hand, it is very important to study weighted estimates for these operators in harmonic analysis. It is well known that $M$ is a bounded operator on $L^{p}(w)$\cite{Mu} with
$w\in A_{p}$, $1<p<\infty$. For any non-negative locally functions
$w$ and any Lebesgue measurable function $f$, we set
$$\|f\|_{L^{p}(w)}=\left(\int_{\mathbb{R}^{n}}|f(x)|^{p}w(x) dx\right)^{1/p}$$
and if $w\equiv1$, we denote $\|f\|_{L^{p}(w)}$ simply by
$\|f\|_{L^{p}(\mathbb{R}^{n})}$. For the weighted $L^{p}(w)$
estimates and weighted weak (1.1) type estimates for $T$, see
\cite{G}. In \cite{MW1}, the authors obtained the corresponding
weighted boundedness on weighted $L^{p}$ spaces for $I_{\alpha}$
with $w\in A_{(p,q)}(1\leq p, q<\infty)$. Here and after,
$A_{p}(1\leq p<\infty)$ and $A_{(p,q)}(1\leq p, q<\infty)$ denote
the Muckenhoupt classes \cite{Mu}.

In \cite{KS}, Komori and Shirai introduced a weighted Morrey
space, which is a natural generalization of weighted Lebesgue
space, and investigated the boundedness of classical operators in
harmonic analysis, that is, the H-L maximal operator $M$, the C-Z
singular integral operator $T$ and the fractional integral
$I_{\alpha}$. Let $1\leq p<\infty,$ $0<\lambda<1$ and $w$ be a
function. Then the weighted Morrey space $M_{p,\lambda}(w)$ is
defined by
$$M_{p,\lambda}(w)=\left\{f: \|f\|_{M_{p,\lambda}(w)}=\sup_{B}\left(\frac{1}{w(B)^{\lambda}}\int_{B}|f(x)|^{p}w(x)dx\right)^{\frac{1}{p}}<\infty\right\}.$$
It is obviously that if $w=1, \lambda=1-\frac{p}{q}$, then
$M_{p,\lambda}(w)=M_{p,q}(\mathbb{R}^{n})$. For $w\in A_{p}(1\leq
p<\infty)$, if $\lambda=0,$ then $M_{p,0}(w)=L^{p}(w)$ and if
$\lambda=1, $ $M_{p,1}(w)=L^{\infty}(w)$.

In the fractional case, we need to consider a weighted Morrey
space with two weights which also introduced by Komori and Shirai
in \cite{KS}. Let $1\leq p<\infty$, $0<\lambda<1$. For two weights
$w_{1}$ and $w_{2}$,
$$
M_{p,\lambda}(w_{1},w_{2})=\left\{
f:\|f\|_{M_{p,k}(w_{1},w_{2})}=\sup_{B}\left(\frac{1}{w_{2}(B)^{\lambda}}\int_{B}|f(x)|^{p}w_{1}(x)dx\right)^{\frac{1}{p}}<\infty\right\}.
$$
If $w_{1}=w_{2}=w$, then we denote
$M_{p,\lambda}(w_{1},w_{1})=M_{p,\lambda}(w_{2},w_{2})=M_{p,\lambda}(w)$.

In \cite{W}, Wang obtained some estimates for Bochner-Riesz means
operators on $M_{p,\lambda}(w)$ by the similar method as in
\cite{KS}. In this paper, we shall establish some boundedness for
some sublinear operators on $M_{p,\lambda}(w)$ and
$M_{p,\lambda}(w_{1},w_{2})$, which includes, as particular cases,
the known results in \cite{KS} and \cite{W}. Applications to the strong solutions
of nondivergence elliptic equations with VMO coefficients are also
given.

Let $D_{k}=\{x\in \mathbb{R}^{n}:|x|\leq 2^{k}\}$ and $A_{k}=D_{k}/D_{k-1}$ for $k\in Z$. Let $\chi_{k}=\chi_{A_{k}}$ for $k\in Z$, where $\chi_{E}$ is the characteristic function of the set $E$. Our mean results are as follows:
\begin{theorem} \label{Theorem 1.1}
Suppose that a sublinear operator $\mathcal{T}$ satisfies the size conditions
\begin{equation}\label{(1.1)}
|\mathcal{T}f(x)|\leq C\|f\|_{L^{1}(\mathbb{R}^{n})}/|x|^{n},
\end{equation}
when $\supp f\subseteq A_{k}$ and $|x|\geq 2^{k+1}$ with $k\in Z$ and
\begin{equation}\label{(1.2)}
|\mathcal{T}f(x)|\leq C2^{-kn}\|f\|_{L^{1}(\mathbb{R}^{n})},
\end{equation}
when $\supp f\subseteq A_{k}$ and $|x|\leq 2^{k-1}$ with $k\in Z$. Then we have\\
$(a)$\,\,If $\mathcal{T}$ is bounded on $L^{p}(w)$ with $w\in A_{p}(1<p<\infty)$, then $\mathcal{T}$ is bounded on $M_{p,\lambda}(w)$.\\
$(b)$\,\,If $\mathcal{T}$ is bounded from $L^{1}(w)$ to $L^{1,\infty}(w)$ with $w\in A_{1}$, then there exist constant $C>0$ such that for all $\mu>0$ and all $B$,
$$
w(\{x\in B:\mathcal{T}f(x)>\mu\})\leq
{C}/{\mu}\|f\|_{M_{1,\lambda}(w)}w(B)^{\lambda}.
$$
\end{theorem}
It is easy to check that the H-L maximal function $M(f)$ satisfies
the hypotheses of Theorem \ref{Theorem 1.1}.

We say $b$ is a $BMO$ function, which means
$\|b\|_{BMO}=\|b^{\sharp}\|_{L^{\infty}}<\infty$, where
$b^{\sharp}(x)$ is sharp maximal function
$$b^{\sharp}(x)=\sup_{ B}\frac{1}{|B|}\int_{B}\left|f(y)-f_{B}\right|dy,$$
where the supreme is taken over all balls $B\subset \mathbb{R}^{n}$ and $f_{B}=\frac{1}{|B|}\int_{B}f(y)dy.$
For $1<p<\infty$, there is a close relation between $BMO$ and
$A_{p}$ weights
$$BMO=\left\{\alpha \log w: w\in A_{p}, \alpha\geq 0\right\}.$$

Given a operator $N$ acting on functions and given a function $b$,
the commutator $[b,N]$ is formally defined as
$$N_{b}f=[b,N]f=bN(f)-N(bf).$$
There is a great amount of works that deal with the topic of
commutators of different operators with $BMO$ functions on
Lesbugue spaces. The first results on this commutator were
obtained by Coifman, Rochberg and Weiss \cite{CRW} in their study
of certain factorization theorems for generalized Hardy spaces.
They show that $N_{b}f$ is bounded on $L^{p}(\mathbb{R}^{n}),
1<p<\infty$, if and only if $b\in BMO$ when $N$ is a classical
singular integral operator with smooth kernel. For some classical
weighted boundedness of $N_{b}$ on $L^{p}(w)$ with $w\in A_{p},
1<p<\infty$, see e.g. \cite{C1}, \cite{C2}. It is well known that
the commutators formed by $BMO$ functions and the fractional
integral $I_{\alpha}$, the C-Z singular integral $T$ are all
bounded on weighted $L^{p}$ spaces \cite{ST}. In \cite{KS}, the
authors also established the weighted boundedness for $T_{b}$ and
$I_{\alpha,b}$ on weighted Morrey spaces. In this paper, we extend
the results of \cite{KS} and obtain

\begin{theorem} \label{Theorem 1.2}
Let $1<p<\infty$, $w\in A_{p}$ and a sublinear operator
$\mathcal{\overline{T}}$ satisfies the conditions
\begin{equation}\label{(1.3)}
|\mathcal{\overline{T}}f(x)|\leq
C\int_{\mathbb{R}^{n}}\frac{|f(y)|}{|x-y|^{n}}dy, \,\,\,x\notin
\emph{supp} f
\end{equation}
for any integral function $f$ with compact support. Then we have\\
$(a)$\,\,If $\mathcal{\overline{T}}$ is bounded on $L^{p}(w)$, then $\mathcal{\overline{T}}$ is bounded on $M_{p,\lambda}(w)$.\\
$(b)$\,\,If $\mathcal{\overline{T}}_{b}$ is bounded $L^{p}(w)$
with $b\in BMO(\mathbb{R}^{n})$, then $\mathcal{\overline{T}}_{b}$
is bounded on $M_{p,\lambda}(w)$
\end{theorem}
It is worth pointing out that (\ref{(1.3)}), which implies the
size conditions in Theorem \ref{Theorem 1.1}, is satisfied by many
operators in harmonic analysis, such as C-Z singular integral
operator, the Carleson maximal operator, C. Fefferman's singular
multiplier operator, R. Fefferman's singular integral operator and so on, see e.g.
\cite{RS} and \cite{SW}.

Besides the H-L maximal operators and C-Z singular integral
operators, oscillatory integral operators have been an essential
part of harmonic analysis; three chapters are devoted to them in
the celebrated Stein's book \cite{St}. Many important operators in
harmonic analysis are some versions of oscillatory integrals, such
 as the Fourier transform, the Bochner-Riesz means, the Radon
transform \cite{PS} in CT technology and so on. For a more
complete account on oscillatory integrals in classical harmonic
analysis, we would like to refer the interested reader to
\cite{L1}, \cite{L2}, \cite{LDY} and references
therein. Another early impetus for the study of oscillatory
integrals came with their application to number theory \cite{Bou}.
In more recent times, the operators fashioned from oscillatory
integrals, such as pseudo-differential operator in PDE become
another motivation to study them. Based on the estimates of some
kinds of oscillatory integrals, one can establish the
well-posedness theory of a class of dispersive equations, for some
of this works, we refer to \cite{CM}, \cite{KPV1} and \cite{KPV2}.

In 1987, Ricci and Stein \cite{RS} introduced one class of oscillatory integrals
$$T_{O}f(x)=\mathrm{p.v.}\int_{\mathbb{R}^{n}}e^{iP(x,y)}K(x-y)f(y)\,dy,$$
which initially defined for smooth function $f$ with compact
support. Here $P(x,y)$ is a real valued polynomial defined on
$\mathbb{R}^{n}\times\mathbb{R}^{n}$, and $K\in
C^{1}(\mathbb{R}^{n}\setminus\{0\})$ is a C-Z kernel. In
\cite{RS}, Ricci and Stein established the boundedness of $T_{O}$
on $L^{p}(\mathbb{R}^{n})(1<p<\infty)$. In 1992, Lu and Zhang
\cite{LZ} gave the boundedness of $T_{O}$ on
$L^{p}(w)(1<p<\infty)$ with $w\in A_{p}$. For the case $p=1$,
Chanillo and Christ \cite{CC} gave a weak (1,1) type estimates
while Sato obtained the weighted version of \cite{CC} in
\cite{Sat}. It is easy to see that $T_{O}$ satisfies Theorem
\ref{Theorem 1.2}.

The Bochner-Riesz mean operators of order $\delta>0$ in $\mathbb{R}^{n}(n\geq 2)$ are defined initially for Schwartz functions in terms of Fourier transforms by
$$
(T_{R}^{\delta}f)^{\wedge}(\xi)=(1-\frac{|\xi|^{2}}{R^{2}})_{+}^{\delta}\hat{f}(\xi),
$$
where $\hat{f}$ denotes the Fourier transform of $f$. These operators were first introduced by Bochner \cite{Boc} in connection with summation of multiple Fourier series and played an important role in harmonic analysis. $T_{R}^{\delta}$ can be expressed as convolution operators $T_{R}^{\delta}f(x)=(f\ast\phi_{\frac{1}{R}})(x)$, where $\phi(x)=[(1+|\cdot|^{2})_{+}^{\delta}]^{\wedge}(x)$. It is well know that the kernel $\phi$ can be represented as \cite{LW}:
$$\phi(x)=\pi^{-\delta}\Gamma(\delta+1)|x|^{-(n/2+\delta)}J_{n/2+\delta}(2\pi|x|),$$
where $J_{\mu}(t)$ is the Bessel function. In \cite{SS}, Shi and
Sun obtained the weighted $L^{p}$ boundedness of $T_{R}^{\delta}$
while Vargas given the weighted weak (1,1) type estimates in
\cite{V}. By the well known boundedness criterion for the
commutators of linear operators, which was obtained by Alvarez,
Bagby, Kurtz and P\'{e}rez \cite{ABKP}, we see that the commutator
$T_{R,b}^{\delta}$ is also bounded on $L^{p}(w)$ with $w\in A_{p}$
for all $1<p<\infty.$ From the asymptotic properties of the Bessel
function, we can deduce that when $\delta=(n-1)/2$, the critical
index, $|\phi(x)|\leq \frac{C}{(1+|x|)^{n}}$, which implies that
$T_{R}^{\delta}$ satisfies the condition of Theorem \ref{Theorem
1.2}.

Given a positive real number $0<a\neq 1$, the oscillating kernel $K_{a}$ is defined by \cite{CKS2}:
$$
K_{a}(x)=(1+|x|)^{-1}e^{i|x|^{a}}.
$$
The convolution operator $T_{a}=K_{a}\ast f$ and closely related weakly singular operators and multiplier operators have been studied by many authors, see e.g. \cite{CKS1}, \cite{DNS} and \cite{JS}. In \cite{CKS2}, Chanillo, Kurtz and Sampson had given the weighted weak (1,1) type estimates and weighted $L^{p}(1<p<\infty)$ estimates for $T_{a}$. It is easy to see that $T_{a}$ satisfies Theorem \ref{Theorem 1.2}.

\begin{theorem} \label{Theorem 1.3}
Let $0<\alpha<n$, $1<p<\frac{n}{\alpha}$ ,$\frac{1}{q}=\frac{1}{p}-\frac{\alpha}{n}$ and $1<p\leq q\leq\infty$. Suppose that a sublinear operator $\mathcal{T}_{\alpha}$ satisfies the size conditions
\begin{equation}\label{(1.4)}
|\mathcal{T}_{\alpha}f(x)|\leq C|x|^{-(n-\alpha)}\|f\|_{L^{1}(\mathbb{R}^{n})}
\end{equation}
when $\supp f\subseteq A_{k}$ and $|x|\geq 2^{k+1}$ with $k\in Z$ and
\begin{equation}\label{(1.5)}
|\mathcal{T}_{\alpha}f(x)|\leq C2^{-k(n-\alpha)}\|f\|_{L^{1}(\mathbb{R}^{n})}
\end{equation}
when $\supp f\subseteq A_{k}$ and $|x|\leq 2^{k-1}$ with $k\in Z$. Then we have\\
$(a)$\,\,If $\mathcal{T}_{\alpha}$ maps $L^{p}(w^{p})$ into $L^{q}(w^{q})$ with $w\in A_{(p,q)}$, then $\mathcal{T}_{\alpha}$ is bounded from  $M_{p,\lambda}(w^{p},w^{q})$ to $M_{q,q\lambda/p}(w^{q})$.\\
$(b)$\,\,If $\mathcal{T}_{\alpha}$ is bounded from $L^{1}(w)$ to
$L^{q,\infty}(w^{q})$ with $w\in A_{(1,q)}$, then there exist
constant $C>0$ such that for all $\mu>0$ and all ball $B$,
$$
w^{q}(\{x\in B:\mathcal{T}_{\alpha}f(x)>\mu\})\leq
{C}/{\mu^{q}}\|f\|_{M_{1,\lambda}(w,w^{q})}^{q}w^{q}(B)^{q\lambda}.
$$
\end{theorem}

The fractional maximal operator $M_{\alpha}$ is defined by
$$
M_{\alpha}f(x)=\sup_{B\ni x}\frac{1}{|B|^{1-\alpha/n}}\int_{B}|f(y)|dy, \,\,\,\,0<\alpha<n.
$$
$M_{\alpha}$ satisfies the hypotheses of Theorem \ref{Theorem 1.3}
since the pointwise inequality $M_{\alpha}f(x)\leq
I_{\alpha}(|f|)(x)$ for $0<\alpha<n$.

\begin{theorem} \label{Theorem 1.4}
Let $p,q,\alpha,w$ be as the same as that of Theorem $\ref{Theorem
1.3}$ and a sublinear operator $\mathcal{\overline{T}}_{\alpha}$
satisfies the conditions
\begin{equation}\label{(1.6)}
|\mathcal{T}_{\alpha}f(x)|\leq
C\int_{\mathbb{R}^{n}}\frac{|f(y)|}{|x-y|^{n-\alpha}}dy,
\,\,\,x\notin \emph{supp} f
\end{equation}
for any integral function $f$ with compact support.\\
$(a)$\,\,If $\mathcal{\overline{T}}_{\alpha}$ maps $L^{p}(w^{p})$ into $L^{q}(w^{q})$, then $\mathcal{\overline{T}}_{\alpha}$ is bounded from $M_{p,\lambda}(w^{p},w^{q})$ to $M_{q,q\lambda/p}(w^{q})$.\\
$(b)$\,\,If $\mathcal{\overline{T}}_{\alpha,b}$ maps
$L^{p}(w^{p})$ into $L^{q}(w^{q})$ with $b\in
BMO(\mathbb{R}^{n})$, then $\mathcal{\overline{T}}_{\alpha,b}$ is
bounded from $M_{p,\lambda}(w^{p},w^{q})$ to
$M_{q,q\lambda/p}(w^{q})$.
\end{theorem}
We remarks that fractional integral $I_{\alpha}$ and oscillatory
fractional integral of Ricci and Stein's \cite{RS} are all
examples of operators which satisfies (\ref{(1.6)}). For the
corresponding boundedness in unweighted cases of the sublinear
operators on Herz space, we refer to \cite{HY} and \cite{LY}.

We end this section with the outline of this paper. In Section \ref{section2}, we give
the proofs of Theorem \ref{Theorem 1.1}-Theorem \ref{Theorem 1.4}. Section
\ref{section3} contains some applications of Theorem \ref{Theorem
1.1}-Theorem \ref{Theorem 1.4}.
Throughout this paper, the letter $C$ is used for various
constants, and may change from one occurrence to another. All
balls are assumed to have their sides paralled to the coordinate
axes. $B=B(x_{0},r)$ denotes the ball centered at $x_{0}$ and with
radius $r$ and $\lambda B=B(x_{0},\lambda r)$.

\section{Proofs of the main results}\label{section2}
Our methods are adopted from \cite{FLY} in the case of the
Lebesgue measure and from \cite{KS} dealing with the classical
operators. Before the proof of Theorem \ref{Theorem 1.1}, we give
some properties of $A_{p}$ weights.
\begin{lemma} \,\,\rm{\cite{G}}\label{lemma 3.1}
Let $1\leq p<\infty$, and $w\in A_{p}$. Then the following
statements are true

$\mathrm{(a)}$\,\, There exists a constant $C$ such that
\begin{equation}\label{3.1}
w(2B)\leq Cw(B).
\end{equation}
When $w$ satisfies this condition,
we say $w$ satisfies doubling condition.

$\mathrm{(b)}$\,\,  There exists a constant $C>1$ such that
\begin{equation}\label{3.2}
w(2B)\geq Cw(B).
\end{equation}
When $w$ satisfies this condition,
we say $w$ satisfies reverse doubling condition.

$\mathrm{(c)}$\,\,  There exist two constant $C$ and $r>1$ such
that the following reverse H\"{o}lder inequality holds for every
ball $B\subset \mathbb{R}^{n}$
\begin{equation}\label{3.3}
\left(\frac{1}{|B|}\int_{B}w(x)^{r}dx \right)^{\frac{1}{r}}\leq C\left(\frac{1}{|B|}\int_{B} w(x)dx\right).
\end{equation}

$\mathrm{(d)}$\,\,  For all $\lambda>1,$ we have
\begin{equation}\label{3.4}
w(\lambda B)\leq C\lambda^{np}w(B).
\end{equation}

$\mathrm{(e)}$\,\,  There exist two constant $C$ and $\delta>0$
such that for any measurable set $Q\subset B$
\begin{equation}\label{3.5}
\frac{w(Q)}{w(B)}\leq C\left( \frac{|Q|}{|B|}\right)^{\delta}.
\end{equation}
If $w$ satisfies (\ref{3.5}), we say $w\in A_{\infty}$.
\end{lemma}

\subsection{Proof of Theorem \ref{Theorem 1.1}}\label{section3.1}
Let $1<p<\infty$, $w\in A_{p}$ and $0<\lambda<1$.  We first give
the proof of (a), which
 suffices to show that
\begin{equation}\label{3.6}
\frac{1}{w(B)^{\lambda}}\int_{B}|\mathcal{T}f(x)|^{p}w(x)dx\leq
C\|f\|_{M_{p,\lambda}(w)}^{p}.
\end{equation}
For a fixed ball $B=B(x_{0},r)$, there is no loss of generality in
assuming $r=1$. We decompose
$f=f\chi_{2B}+f\chi_{(2B)^{c}}:=f_{1}+f_{2}$. Since $\mathcal{T}$
is a sublinear operator, so we get
$$
\frac{1}{w(B)^{\lambda}}\int_{B}|\mathcal{T}f(x)|^{p}w(x)dx\leq
\frac{1}{w(B)^{\lambda}}\int_{B}(|\mathcal{T}f_{1}(x)|^{p}+|\mathcal{T}f_{2}(x)|^{p})w(x)dx:=I+II.
$$
Using the fact that $\mathcal{T}$  is bounded on $L^{p}(w)$, we can easily get
\begin{equation}\label{3.7}
I\leq
\frac{1}{w(B)^{\lambda}}\int_{\mathbb{R}^{n}}|\mathcal{T}f_{1}(x)|^{p}w(x)dx\leq
\frac{C}{w(B)^{\lambda}}\int_{2B}|f(x)|^{p}w(x)dx \leq
C\|f\|_{M_{p,\lambda}(w)}^{p}.
\end{equation}

We are now in a position to estimate the term $II$. We  conclude from $w\in A_{p}$ that
\begin{align*}
\int_{(2B)^{c}}|f(y)|dy
&\leq C\sum_{k=1}^{\infty}\int_{2^{k+1}B/2^{k}B}|f(y)|dy\\
&\leq
C\sum_{k=1}^{\infty}\left(\int_{2^{k+1}B}|f(y)|^{p}w(y)dy\right)^{1/p}\left(\int_{2^{k+1}B}w(y)^{-p'/p}dy\right)^{1/p'}\\
&\leq C\|f\|_{M_{p,\lambda}(w)}\sum_{k=1}^{\infty}\frac{|2^{k+1}B|}{w(2^{k+1}B)^{1-\lambda/p}}.
\end{align*}
By (\ref{(1.2)}), we have
\begin{equation}\label{3.8}
\begin{split}
II&\leq \frac{C}{w(B)^{\lambda}}2^{-knp}\int_{B}\|f_{2}\|_{L^{1}(\mathbb{R}^{n})}^{p}w(x)dx
\\
&\leq \frac{C}{w(B)^{\lambda-1}}2^{-knp}\left(\int_{(2B)^{c}}|f(y)|dy\right)^{p}\\
&\leq C\|f\|_{M_{p,\lambda}(w)}^{p}\left(\sum_{k=1}^{\infty}\frac{w(B)^{(1-\lambda)/p}}{w(2^{k+1}B)^{(1-\lambda)/p}}\right)^{p}\\
&\leq C\|f\|_{M_{p,\lambda}(w)}^{p}.
\end{split}
\end{equation}
Here we use (\ref{3.1}) in the last inequality above. Combing
(\ref{3.7}) with (\ref{3.8}), we get (\ref{3.6}), which yields the
proof of (a).

Now, we have a position to give the proof of (b), which is similar to that of (a). We want to set up the following inequality
$$
\sup_{\mu>0}\frac{\mu}{w(B)^{\lambda}}w\left(\{x\in B: |\mathcal{T}f(x)|>\mu\}\right)\leq C\|f\|_{M_{1,\lambda}(w)}^{p}.
$$
Decompose $f=f\chi_{2B}+f\chi_{(2B)^{c}}:=f_{1}+f_{2}$ with $B$ as
that of (a). For any given $\mu>0$, we write
\begin{align*}
w\left(\{x\in B: |\mathcal{T}f(x)|>\lambda\}\right)&\leq w\left(\left\{x\in B: |\mathcal{T}f_{1}(x)|>{\mu}/{2}\right\}\right)+w\left(\left\{x\in B: |\mathcal{T}f_{2}(x)|>{\mu}/{2}\right\}\right)\\
&:= J+JJ.
\end{align*}
An application of (\ref{3.1}) and the weighted weak (1,1) type
estimates for $\mathcal{T}$ yield that
$$
J\leq w\left(\left\{x\in \mathbb{R}^{n}:
|\mathcal{T}f_{1}(x)|>{\mu}/{2}\right\}\right)\leq
{C}/{\mu}\|f\|_{M_{1,\lambda}}(w)w(B)^{\lambda}.
$$

Next we turn to deal with the term $JJ.$ An elementary estimate
shows
$$
JJ\leq \frac{C}{\mu}\int_{\left\{x\in B: |\mathcal{T}f(x)|>\frac{\mu}{2}\right\}}|\mathcal{T}f_{2}(x)|w(x)dx.
$$
Applying $(\ref{(1.2)})$, we conclude that
$$
|\mathcal{T}f_{2}(x)|
\leq C2^{-kn}\int_{(2B)^{c}}|f(y)|dy\leq C\sum_{k=1}^{\infty}2^{-kn}\int_{2^{k+1}B}|f(y)|dy.
$$
H\"{o}lder's inequality and the $A_{1}$ condition imply that
\begin{align*}
JJ&\leq
\frac{C}{\mu}\sum_{k=1}^{\infty}2^{-kn}\int_{2^{k+1}B}|f(y)|w(y)dy\\
&\leq\frac{C}{\mu}\|f\|_{M_{1,k}}(w)\sum_{k=1}^{\infty}2^{kn(\lambda-1)}w(B)^{\lambda}\\
&\leq \frac{C}{\mu}\|f\|_{M_{1,k}}(w)w(B)^{\lambda}.
\end{align*}
Then, we have completed the proof of (b).
\qed
\subsection{Proof of Theorem \ref{Theorem 1.2}}\label{section3.2}
The proof of Theorem \ref{Theorem 1.2} depend heavily on the
following remarks about $BMO$ functions.

\begin{lemma}\,\, \rm{\cite{T}}\label{lemma 3.2} \quad
Let $1\leq p<\infty$, $b\in BMO(\mathbb{R}^{n})$. Then for any
ball $B\subset \mathbb{R}^{n}$, the following statements are true

$\mathrm{(a)}$\,\, There exist constants $C_{1}$, $C_{2}$ such
that for all $\alpha>0$
\begin{equation}\label{3.9}
\left|\{x\in B:|b(x)-b_{B}|>\alpha\}\right|\leq C_{1}|B|e^{-C_{2}\alpha/\|b\|_{BMO(\mathbb{R}^{n})}}.
\end{equation}
The inequality $(\ref{3.9})$ is also called John-Nirenberg inequality.

$\mathrm{(b)}$\,\,
\begin{equation}\label{3.10}
|b_{2^{\lambda}B}-b_{B}|\leq2^{n}\lambda \|b\|_{BMO(\mathbb{R}^{n})}.
\end{equation}
\end{lemma}
\begin{lemma} \,\,\rm{ \cite{MW2}}\label{lemma 3.3}
 Let $w\in A_{\infty}$. Then the following statements are equivalent

$\mathrm{(a)}$\,\,
\begin{equation}\label{3.11}
\|b\|_{BMO(\mathbb{R}^{n})}\sim \sup_{B}\left(\frac{1}{|B|}\int_{B}|b(x)-b_{B}|^{p}dx\right)^{\frac{1}{p}}.
\end{equation}

$\mathrm{(b)}$\,\,
\begin{equation}\label{3.12}
\|b\|_{BMO(\mathbb{R}^{n})}\sim \sup_{B}\inf_{a\in \mathbb{R}}\frac{1}{|B|}\int_{B}|b(x)-a|dx.
\end{equation}

$\mathrm{(c)}$\,\,
\begin{equation}\label{3.13}
\|b\|_{BMO(w)}=\sup_{B}\frac{1}{w(B)}\int_{B}|b(x)-b_{B,w}|w(x)dx.
\end{equation}
where $BMO(w)=\{b:\|b\|_{BMO(w)}<\infty\}$ and
$b_{B,w}=\frac{1}{w(B)}\int_{B}b(y)w(y)dy.$
\end{lemma}
\begin{lemma}\label{lemma 3.4}
Let $b\in BMO(\mathbb{R}^{n})$, $B=B(x_{0},r)$, $0<\lambda<1$ and
$1<p<\infty$. Then the inequality
\begin{equation}\label{3.14}
\left(\int_{|x_{0}-y|>2r}\frac{|f(y)|}{|x_{0}-y|^{n}}|b_{B,w}-b(y)|dy\right)^{p}w(B)^{1-\lambda}\leq
C\|f\|_{M_{p,\lambda}(w)}^{p}\|b\|_{BMO(\mathbb{R}^{n})}^{p}.
\end{equation}
holds for every $y\in (2B)^{c}$, where $(2B)^{c}=\mathbb{R}^{n}/2B.$
\end{lemma}
\begin{proof}
Using H\"{o}lder's inequality to the left-hand-side of $(\ref{3.14})$,
we have
\begin{align*}
&\left(\int_{|x_{0}-y|>2r}\frac{|f(y)|}{|x_{0}-y|^{n}}|b_{B,w}-b(y)|dy\right)^{p}w(B)^{1-\lambda}\\
&\leq \left(\sum_{j=1}^{\infty}\int_{2^{j}r<|x_{0}-y|<2^{j+1}r}\frac{|f(y)|}{|x_{0}-y|^{n}}|b_{B,w}-b(y)|dy\right)^{p}w(B)^{1-\lambda}\\
&\leq \left(\sum_{j=1}^{\infty}\frac{1}{|2^{j}B|}\int_{2^{j+1}B}|f(y)||b_{B,w}-b(y)|dy\right)^{p}w(B)^{1-\lambda}\\
&\leq C\left[\sum_{j=1}^{\infty}\frac{1}{|2^{j}B|}\left(\int_{2^{j+1}B}|f(y)|^{p}w(y)dy\right)^{\frac{1}{p}}\left(\int_{2^{j+1}B}|b_{B,w}-b(y)|^{p'}w(y)^{1-p'}dy\right)^{\frac{1}{p'}}\right]^{p}w(B)^{1-\lambda}\\
&\leq
C\|f\|_{M_{p,\lambda}(w)}^{p}\left[\sum_{j=1}^{\infty}\frac{w(2^{j+1}B)^{\frac{\lambda}{p}}}{|2^{j}B|}\left(\int_{2^{j+1}B}|b_{B,w}-b(y)|^{p'}w(y)^{1-p'}dy\right)^{\frac{1}{p'}}\right]^{p}w(B)^{1-\lambda}.
\end{align*}

For the simplicity of analysis, we denote $A$ as
$$
\left(\int_{2^{j+1}B}|b_{B,w}-b(y)|^{p'}w(y)^{1-p'}dy\right)^{\frac{1}{p'}}.
$$
By an elementary estimate, we have
\begin{align*}
A
&\leq \left(\int_{2^{j+1}B}(|b_{2^{j+1}B,w^{1-p'}}-b(y)|+|b_{2^{j+1}B,w^{1-p'}}-b_{B,w}|)^{p'}w(y)^{1-p'}dy\right)^{\frac{1}{p'}}\\
&\leq \left\|\frac{|b_{2^{j+1}B,w^{1-p'}}-b(\cdot)|+|b_{2^{j+1}B,w^{1-p'}}-b_{B,w}|}{w(\cdot)}\right\|_{L^{p'}(w)}\\
&\leq \left(\int_{2^{j+1}B}|b_{2^{j+1}B,w^{1-p'}}-b(y)|w(y)^{1-p'}dy\right)^{\frac{1}{p'}}+|b_{2^{j+1}B,w^{1-p'}}-b_{B,w}|w^{1-p'}(2^{j+1}B)^{\frac{1}{p'}}\\
&= :A_{1}+A_{2}.
\end{align*}

For the term $A_{1}$, Lemma \ref{lemma 3.3} implies
\begin{equation}\label{3.15}
A_{1}\leq
C\|b\|_{BMO(w^{1-p'})}w^{1-p'}(2^{j+1}B)^{\frac{1}{p'}}\leq
Cw^{1-p'}(2^{j+1}B)^{\frac{1}{p'}}.
\end{equation}
To deal with $A_{2}$, by $(\ref{3.10})$, we have
\begin{equation*}
\begin{split}
|b_{2^{j+1}B,w^{1-p'}}-b_{B,w}|
&\leq |b_{2^{j+1}B,w^{1-p'}}-b_{2^{j+1}B}|+|b_{2^{j+1}B}-b_{B}|+|b_{B}-b_{B,w}|\\
&\leq \frac{1}{w^{1-p'}(2^{j+1}B)}\int_{2^{j+1}B}|b(y)-b_{2^{j+1}B}|w(y)^{1-p'}dy+2^{n}(j+1)\|b\|_{BMO(\mathbb{R}^{n})}\\
&\quad\quad\quad\quad+\frac{1}{w(B)}\int_{B}|b(y)-b_{B}|w(y)dy\\
&:=A_{21}+A_{22}+A_{23}.
\end{split}
\end{equation*}
Combining $(\ref{3.5})$ with $(\ref{3.9})$, we have
\begin{align*}
A_{23}
&= \frac{1}{w(B)}\int_{0}^{\infty}w(\{x\in B:|b(y)-b_{B}|>\alpha \})d\alpha\\
&\leq C\int_{0}^{\infty}e^{-C_{2}\alpha\delta/\|b\|_{BMO(\mathbb{R}^{n})}}d\alpha\\
&\leq C.
\end{align*}
In the same manner we can see that
$$
A_{21}\leq C.
$$
It follows immediately that
\begin{equation}\label{3.16}
A_{2}\leq C(2^{n}(j+1)+2)w^{1-p'}(2^{j+1}B)^{\frac{1}{p'}}.
\end{equation}
As a by-product of $(\ref{3.15})$ and $(\ref{3.16})$, we have
$$
A\leq C(j+1)w^{1-p'}(2^{j+1}B)^{\frac{1}{p'}}.
$$
Then, applying $(\ref{3.2})$, the proof of $(\ref{3.14})$ based on the
following observation
\begin{align*}
&\left[\sum_{j=1}^{\infty}\frac{w(2^{j+1}B)^{\frac{k}{p}}}{|2^{j}B|}\left(\int_{2^{j+1}B}|b(y)-b_{B,w}|^{p'}w(y)^{1-p'}dy\right)^{\frac{1}{p'}}\right]^{p}
w(B)^{1-k}\\
&\leq
C\left[\sum_{j=1}^{\infty}\frac{w(B)^{\frac{1-k}{p}(j+1)}}{w(2^{j+1}B)^{\frac{1-k}{p}}}\right]^{p}=C.
\end{align*}
\end{proof}

Now, we come back to the proof of Theorem \ref{Theorem 1.2}. (a)
is trivial since (\ref{(1.3)}) satisfies Theorem \ref{Theorem
1.1}. We only need to give the proof of (b).  The task is now to
find a constant $C$ such that for fixed ball $B=B(x_{0},1)$, we
can obtain
\begin{equation}\label{3.17}
\frac{1}{w(B)^{\lambda}}\int_{B}\left|\mathcal{\overline{T}}_{b}f(x)\right|^{p}w(x)dx\leq
C\|f\|_{M_{p,\lambda}(w)}^{p}.
\end{equation}

We decompose $f=f\chi_{2B}+f\chi_{(2B)^{c}}:=f_{1}+f_{2},$ and
consider the corresponding splitting
\begin{align*}
\int_{B}\left|\mathcal{\overline{T}}_{b}f(x)\right|^{p}w(x)dx&\leq C\left(\int_{B}|\mathcal{\overline{T}}_{b}f_{1}(x)|^{p}w(x)dx+\int_{B}|\mathcal{\overline{T}}_{b}f_{2}(x)|^{p}w(x)dx\right)\\
&=:K+KK.
\end{align*}

It follows from the $L^{p}(w)$ boundedness of
$\mathcal{\overline{T}}_{b}$ and $w\in A_{p}$ that
\begin{equation}\label{3.18}
K\leq C\int_{2B}|f(x)|^{p}w(x)dx\leq
C\|f\|_{M_{p,\lambda}(w)}^{p}w(B)^{\lambda}.
\end{equation}
 Then a further use of
$(\ref{(1.3)})$ derives that
\begin{align*}
\left|\mathcal{\overline{T}}_{b}f_{2}(x)\right|^{p}
&\leq C\left(\int_{\mathbb{R}^{n}}\frac{|f_{2}(y)||b(x)-b(y)|}{|x-y|^{n}}dy\right)^{p}\\
&\leq
C\left(\int_{|x_{0}-y|>2}\frac{|f(y)|}{|x_{0}-y|^{n}}\{|b(x)-b_{B,w}|+|b_{B,w}-b(y)|\}dy\right)^{p}.
\end{align*}
where $b_{B,w}=\frac{1}{w(B)}\int_{B}b(x)w(x)dx$. Then, we have
\begin{align*}
KK&\leq C\left(\int_{|x_{0}-y|>2}\frac{|f(y)|}{|x_{0}-y|^{n}}dy\right)^{p}\int_{B}|b(x)-b_{B,w}|^{p}w(x)dx\\
&\quad\quad\quad\quad+C\left(\int_{|x_{0}-y|>2}\frac{|f(y)|}{|x_{0}-y|^{n}}|b(y)-b_{B,w}|dy\right)^{p}w(B)\\
&:= KK_{1}+KK_{2}.
\end{align*}
A further use of Lemma \ref{lemma 3.4}, we get
$$KK_{2}\leq C\|f\|_{M_{p,\lambda}(w)}^{p}w(B)^{\lambda}.$$

To get the desired estimate, we are led to estimate the term
$KK_{1}$. This estimate will be done via (\ref{3.1}), (\ref{3.3})
and Lemma \ref{lemma 3.3}. In fact,

\begin{equation*}
\begin{split}
KK_{1}&=\left(\sum_{j=1}^{\infty}\int_{2^{j}<|x_{0}-y|<2^{j+1}}\frac{|f(y)|}{|x_{0}-y|^{n}}dy\right)^{p}\int_{B}|b(x)-b_{B,w}|^{p}w(x)dx\\
&\leq\left(\sum_{j=1}^{\infty}\frac{1}{|2^{j}B|}\int_{2^{j+1}B}|f(y)|dy\right)^{p}\int_{B}|b(x)-b_{B,w}|^{p}w(x)dx\\
&\leq C\sum_{j=1}^{\infty}\frac{1}{|2^{j}B|}\left(\frac{1}{w(2^{j+1}B)^{\lambda}}\int_{2^{j+1}B}|f(y)|^{p}w(y)dy\right)^{{1}/{p}}\\
&\quad\quad\quad\quad\times w(2^{j+1}B)^{{\lambda}/{p}}\left(\int_{2^{j+1}B}w(y)^{-{1}/{p-1}}dy\right)^{\frac{p-1}{p}}\int_{B}|b(x)-b_{B,w}|^{p}w(x)dx\\
&\leq C\|f\|_{M_{p,\lambda}(w)}\left(\sum_{j=1}^{\infty}\frac{|2^{j+1}B|^{-\frac{1}{p}}}{|2^{j}B|}\left(\frac{1}{|2^{j+1}B|}\int_{2^{j+1}B}w(y)dy\right)^{-{1}/{p}}w(2^{j+1}B)^{{\lambda}/{p}}\right)^{p}\\
&\quad\quad\quad\quad\times \int_{B}|b(x)-b_{B,w}|^{p}w(x)dx\\
&\leq C\|f\|_{M_{p,\lambda}(w)}^{p}\|b\|_{BMO(\mathbb{R}^{n})}^{p}\sum_{j=1}^{\infty}\left(\frac{w(B)^{{(1-\lambda)}/{p}}}{w(2^{j+1}B)^{{(1-k)}/{p}}}\right)^{p}w(B)^{\lambda}\\
&\leq C\|f\|_{M_{p,\lambda}(w)}^{p}w(B)^{\lambda}
\end{split}
\end{equation*}
Hence
\begin{equation}\label{3.19}
KK\leq C\|f\|_{M^{p,\lambda}(w)}^{p}w(B)^{\lambda}.
\end{equation}

Combing (\ref{3.18}), (\ref{3.19}), we obtain (\ref{3.17}), which is the
desired conclusion.\qed

\subsection{Proof of Theorem \ref{Theorem 1.3}}\label{section3.3}

We can use the similar argument as the proof of Theorem
\ref{Theorem 1.1}. For the proof of (a), it suffices to show that
\begin{equation}\label{3.20}
\frac{1}{w^{q}(B)^{q\lambda/p}}\int_{B}|\mathcal{T}_{\alpha}f(x)|^{q}w(x)^{q}dx\leq
C\|f\|_{M_{p,\lambda}(w^{p},w^{q})}^{q}.
\end{equation}
For a fixed ball $B=B(x_{0},1)$, we decompose
$f=f\chi_{2B}+f\chi_{(2B)^{c}}:=f_{1}+f_{2}$.
Since $\mathcal{T}_{\alpha}$ is a sublinear operator, so we get
$$
\frac{1}{w^{q}(B)^{q\lambda/p}}\int_{B}|\mathcal{T}_{\alpha}f(x)|^{q}w(x)^{q}dx\leq
\frac{1}{w^{q}(B)^{q\lambda/p}}\int_{B}(|\mathcal{T}_{\alpha}f_{1}(x)|^{q}+|\mathcal{T}_{\alpha}f_{2}(x)|^{q})w^{q}(x)dx:=L+LL.
$$
To estimate the term $L$, using the fact that $\mathcal{T}_{\alpha}$  is bounded from $L^{p}(w^{p})$ to $L^{q}(w^{q})$ with $w\in A_{(p,q)}$, we can get
$$
\int_{B}|\mathcal{T}_{\alpha}f_{1}(x)|^{q}w^{q}(x)dx\leq
C\|f\|_{M_{p,\lambda}(w^{p},w^{q})}^{q}w^{q}(B)^{q\lambda/p},
$$
which implies that
$$
L\leq C\|f\|_{M_{p,\lambda}(w^{p},w^{q})}.
$$

For the term $LL$. By the similar argument as that of Theorem \ref{Theorem 1.1}, we obtain
\begin{align*}
LL&\leq C\sum_{k}\left(2^{-k(n-\alpha)}\int_{A_{k}}|f(y)|dy\right)^{q}w^{q}(B)^{1-q\lambda/p}
\\
&\leq C\sum_{k}\left(2^{-k(n-\alpha)}\|f\|_{M_{p,\lambda}(w^{p},w^{q})}|2^{k+1}B|^{1-\alpha/n}\frac{1}{w^{q}(2^{k+1}B)^{1/q-\lambda/p}}\right)^{q}w^{q}(B)^{1-q\lambda/p}
\\
&\leq C\|f\|_{M_{p,\lambda}(w^{p},w^{q})}^{q}\left(\sum_{k=1}^{\infty}\frac{w^{q}(B)^{(1/q-\lambda/p)}}{w^{q}(2^{k+1}B)^{(1/q-\lambda/p)}}\right)^{q}\\
&\leq C\|f\|_{M_{p,\lambda}(w^{p},w^{q})}^{q}.
\end{align*}
We have completed the proof of (a).

We shall omit the proof of (b) since we can prove by using $A_{(1,q)}$ condition and the weak type estimates of $\mathcal{T}_{\alpha}$.
\qed

\subsection{Proof of Theorem \ref{Theorem 1.4}}\label{section3.4}
The proof of Theorem \ref{Theorem 1.4} is similar to that of
Theorem \ref{Theorem 1.2}, except using $w\in A_{(p,q)}$.

\section{Applications}\label{section3}
In this section, we shall give some applications of our main results to nondivergence elliptic equations.
Dirichlet problem on the second order elliptic equation in nondivergence form is

\begin{equation}\label{(2.1)}
\left\{\begin{array}{ll}
&Lu=\sum_{i,j}^{n}a_{ij}(x)u_{x_{i}}u_{x_{j}}=f\,\quad a.e\quad in \,\,\,\,\Omega,\\
&u=0\,\quad on\quad \partial\Omega.
\end{array}\right.
\end{equation}
Here $x=(x_{1},\cdots,x_{n})\in \mathbb{R}^{n}$, $\Omega$ is a
bounded domain of $\mathbb{R}^{n}$. The coefficients
$(a_{ij})_{i,j=1}^{n}$ of $L$ are symmetric and uniformly
elliptic, i.e., for some $\nu\geq 1$ and every $\xi\in
\mathbb{R}^{n}$, $a_{ij}(x)=a_{ji}(x)$ and $\nu^{-1}|\xi|^{2}\leq
\sum_{i,j=1}^{n}a_{ij}(x)\xi_{i}\xi_{j}\leq\nu|\xi|^{2}$ with
a.e.$x\in \Omega$. In \cite{FLY}, Fan, Lu and Yang investigate the
regularity in $M_{p,\lambda}(\Omega)$ of the strong solution to
(\ref{(2.1)}) with $a_{ij}\in VMO(\Omega)$, the space of the
functions of vanishing mean oscillation introduced by Sarason in
\cite{Sar}. The main method of \cite{FLY} is based on integral
representation formulas established in \cite{CFL1} and \cite{CFL2}
for the second derivatives of the solution $u$ to (\ref{(2.1)}),
and on the theories of singular integrals and sublinear
commutators in Morrey spaces.

By extending some theorems of \cite{FLY} to weighted versions, we
can establish regularity in weighted Morrey spaces of strong
solutions to nondivergence elliptic equations with $VMO$
coefficients. Theorem \ref{Theorem 1.2} is just the weighted
version of the important theorem-Theorem 2.1 in \cite{FLY}. For
the complement of our paper, we take another important
theorem-Theorem 2.3 of \cite{FLY} to state this. All other proofs
of the corresponding theorems are straightforward.

Let
$\mathbb{R}_{+}^{n}=\{x=(x',x_{n}):x'=(x_{1},\cdots,x_{n-1})\in
\mathbb{R}^{n-1}, x_{n}>0\}$,
$L^{p}_{+}(w)=L^{p}(w,\mathbb{R}_{+}^{n})$ and
$M_{p,\lambda}^{+}=M_{p,\lambda}(w,\mathbb{R}_{+}^{n})$. To
establish the boundary estimates of the solutions to
(\ref{(2.1)}), we need the following general theorem for sublinear
operators.
\begin{theorem}\,\,\,\,\label{theorem 2.8}
Let $1<p<\infty$, $0<\lambda<1$, $w\in A_{p}$,
$\tilde{x}=(x',-x_{n})$ for $x=(x',x_{n})\in \mathbb{R}_{+}^{n}$.
If a sublinear operator $\mathfrak{T}$ is bounded on
$L^{p}_{+}(w)$ for any $f\in L^{1}_{+}(w)$ with compact support
and satisfies
\begin{equation}\label{(2.2)}
|\mathfrak{T}f(x)|\leq C\int_{\mathbb{R}_{+}^{n}}\frac{|f(y)|}{|\tilde{x}-y|^{n}}dy,
\end{equation}
then $\mathfrak{T}$ is bounded on $M_{p,\lambda}^{+}(w)$.
\end{theorem}
\begin{proof}
Let $z\in \mathbb{R}_{+}^{n}$ and $\delta>0$. Set $B_{\delta}^{+}(z)=B_{\delta}(z)\cap \mathbb{R}_{+}^{n}$, where $B_{\delta}(z)=\{y\in \mathbb{R}^{n}:|z-y|<\delta\}$. We consider two cases\\
Case 1.\,\,\,\, $0\leq z_{n}<2\delta.$ In this case, we write
$$
f(y)=f(y)\chi_{B_{2^{4}\delta}^{+}(z)}(y)+\sum_{i=4}^{\infty}f(y)\chi_{B_{2^{i+1}\delta}^{+}(z)/B_{2^{i}\delta}^{+}(z)}(y)\equiv \sum_{i=3}^{\infty}f_{i}(y).
$$
Therefore, by the $L^{p}_{+}(w)$ boundedness of $\mathfrak{T}$ and
(\ref{(2.2)}), we obtain
\begin{align*}
&\frac{1}{w(B_{\delta}^{+})^{\lambda/p}}\left(\int_{B_{\delta}^{+}}|\mathfrak{T}f(x)|^{p}w(x)dx\right)^{1/p}\\
&\leq \frac{1}{w(B_{\delta}^{+})^{\lambda/p}}\sum_{i=3}^{\infty}\left(\int_{B_{\delta}^{+}}|\mathfrak{T}f_{i}(x)|^{p}w(x)dx\right)^{1/p}\\
&\leq \frac{C}{w(B_{\delta}^{+})^{\lambda/p}}\|f_{3}\|_{L^{p}_{+}(w)}+\frac{C}{w(B_{\delta}^{+})^{\lambda/p}}\sum_{i=4}^{\infty}
\left(\int_{B_{\delta}^{+}}\left(\int_{B_{2^{i+1}\delta}^{+}(z)/B_{2^{i}\delta}^{+}(z)}\frac{|f(y)|}{|\tilde{x}-y|^{n}}dy\right)^{p}w(x)dx\right)^{1/p}\\
&\leq C\|f\|_{M_{p,\lambda}^{+}(w)}+C\sum_{i=4}^{\infty}\frac{1}{(2^{i}\delta)^{n}}
\left(\int_{B_{2^{i+1}\delta}^{+}}|f(y)|dy\right)w(B_{\delta}^{+})^{1-\lambda/p}\\
&\leq C\|f\|_{M_{p,\lambda}^{+}(w)}\left(1+\sum_{i=4}^{\infty}\frac{w(B_{\delta}^{+})^{1-\lambda/p}}{w(B_{2^{i+1}\delta}^{+})^{1-\lambda/p}}\right)
\\
&\leq C\|f\|_{M_{p,\lambda}^{+}(w)}.
\end{align*}
In the last inequality, we use Lemma \ref{lemma 3.1} in Section \ref{section2}.\\
Case 2.\,\,\,\,There exists $i\in \mathbb{N}$ such that
$2^{i}\delta\leq z_{n}<2^{i+1}\delta$. In this case, we write
$$
f(y)=f(y)\chi_{B_{2^{i+1}\delta}^{+}(z)}(y)+\sum_{j=1}^{\infty}f(y)\chi_{B_{2^{i+j+4}\delta}^{+}(z)}(y)\equiv \sum_{j=0}^{\infty}f_{j}(y).
$$
By (\ref{(2.2)}) and Lemma \ref{lemma 3.1}, we have
\begin{align*}
&\frac{1}{w(B_{\delta}^{+})^{\lambda/p}}\left(\int_{B_{\delta}^{+}}|\mathfrak{T}f(x)|^{p}w(x)dx\right)^{1/p}\\
&\leq \frac{C}{w(B_{\delta}^{+})^{\lambda/p}}\left(\int_{B_{\delta}^{+}}\left(\int_{B_{2^{i+4}\delta}^{+}(z)}\frac{|f(y)|}{|\tilde{x}-y|^{n}}dy\right)^{p}w(x)dx\right)^{1/p}\\
&\quad +\frac{C}{w(B_{\delta}^{+})^{\lambda/p}}\sum_{j=1}^{\infty}
\left(\int_{B_{\delta}^{+}}\left(\int_{B_{2^{i+j+4}\delta}^{+}(z)/B_{2^{i+j+3}\delta}^{+}(z)}\frac{|f(y)|}{|\tilde{x}-y|^{n}}dy\right)^{p}w(x)dx\right)^{1/p}\\
&\leq \frac{C}{w(B_{\delta}^{+})^{\lambda/p}}\frac{1}{(2^{i}\delta)^{n}}\left(\int_{B_{\delta}^{+}}\left(\int_{B_{2^{i+4}\delta}^{+}(z)}|f(y)|dy\right)^{p}w(x)dx\right)^{1/p}\\
&\quad +\frac{C}{w(B_{\delta}^{+})^{\lambda/p}}\sum_{j=1}^{\infty}\frac{1}{(2^{i+j}\delta)^{n}}\left(\int_{B_{\delta}^{+}}\left(\int_{B_{2^{i+j+4}\delta}^{+}(z)}|f(y)|dy\right)^{p}w(x)dx\right)^{1/p}\\
&\leq C\|f\|_{M_{p,\lambda}^{+}(w)}\left(\frac{w(B_{\delta}^{+})^{1-\lambda/p}}{w(B_{2^{i+4}\delta}^{+})^{1-\lambda/p}}
+\sum_{j=1}^{\infty}\frac{w(B_{\delta}^{+})^{1-\lambda/p}}{w(B_{2^{i+j}\delta}^{+})^{1-\lambda/p}}\right)\\
&\leq C\|f\|_{M_{p,\lambda}^{+}(w)}.
\end{align*}\end{proof}
\subsection*{Acknowledgements}
The authors thank Professor Heping Liu for the valuable suggestions.

\end{document}